\documentclass{amsart}
\usepackage{amsmath}
\usepackage{amscd}
\usepackage{amssymb}

 \newtheorem{thm}{Theorem}[section]
 \newtheorem{cor}[thm]{Corollary}
 \newtheorem{lem}[thm]{Lemma}
 \newtheorem{prop}[thm]{Proposition}
 \theoremstyle{definition}
 
 \newtheorem{ex}[thm]{Example}
 \newtheorem{rem}[thm]{Remark}
 \newtheorem*{finrem}{Final remark}

\newenvironment{enum}{\parindent0pt%
\begin{list}{}{%
\setlength{\itemindent}{0ex}
\setlength{\labelwidth}{15pt}
\setlength{\labelsep}{6pt}
\setlength{\leftmargin}{21pt}
\setlength{\listparindent}{0pt}
\setlength{\itemsep}{0ex}
\setlength{\topsep}{0ex}
\setlength{\parsep}{0.2em} 
}
}{\end{list}}

\newcommand{\bZ}{\mathbb Z}
\newcommand{\bR}{\mathbb R}
\newcommand{\sbR}{{}^*{\mathbb R}}
\newcommand{\sbZ}{{}^*{\mathbb Z}}
\newcommand{\sbN}{{}^*{\mathbb N}}
\newcommand{\FsR}{{\mathbb F}^*{\mathbb R}}
\newcommand{\IsR}{{\mathbb I}^*{\mathbb R}}
\newcommand{\FL}{{\mathbb F}{\hskip.3pt}L}
\newcommand{\IL}{{\mathbb I}{\hskip.3pt}L}
\newcommand{\FG}{{\mathbb F}{\hskip.4pt}G}
\newcommand{\IG}{{\mathbb I}{\hskip.4pt}G}

\newcommand{\Ast}{{}^{*\!}A}
\newcommand{\Bst}{{}^{*\!}B}
\newcommand{\Lam}{\varLambda}
\newcommand{\eps}{\varepsilon}
\newcommand{\bbeta}{\boldsymbol{\beta}}
\newcommand{\ao}{{}^{\circ}a}
\newcommand{\bo}{{}^{\circ}b}
\newcommand{\ds}{{}^{\circ}d}
\newcommand{\xo}{{}^{\circ}x}
\newcommand{\yo}{{}^{\circ}y}
\newcommand{\uo}{{}^{\circ}u}
\newcommand{\Ao}{{}^{\circ\!}A}
\newcommand{\Bo}{{}^{\circ\!}B}
\newcommand{\Co}{{}^{\circ}C}
\newcommand{\Do}{{}^{\circ\!}D}
\newcommand{\Lo}{{}^{\circ}L}
\newcommand{\Po}{{}^{\circ\!}P}
\newcommand{\Qo}{{}^{\circ}Q}
\newcommand{\sbs}{\subseteq}

\newcommand{\sms}{\smallsetminus}
\newcommand{\nin}{\notin}
\newcommand{\cd}{\cdot}
\newcommand{\cx}{\times}
\newcommand{\co}{\circ}
\newcommand{\lv}{\left|}
\newcommand{\rv}{\right|}
\newcommand{\LV}{\left\|}
\newcommand{\RV}{\right\|}
\newcommand{\iso}{\cong}
\newcommand{\apr}{\approx}
\newcommand{\napr}{\not\approx}
\newcommand{\nll}{\not\!\ll}
\renewcommand{\:}{\colon}
\newcommand{\Iff}{\ \Leftrightarrow\ }
\newcommand{\IFF}{\ \Longleftrightarrow\ }
\newcommand{\all}{\forall\,}
\newcommand{\exs}{\exists\,}

\newcommand{\spn}{\operatorname{span}}
\newcommand{\sspn}{{}^*\!\operatorname{span}}
\newcommand{\grp}{\operatorname{grp}}
\newcommand{\sgrp}{{}^*\!\operatorname{grp}}
\newcommand{\rank}{\operatorname{rank}}
\newcommand{\rankf}{\rank_{\mathrm{f}}}
\newcommand{\Anz}{\operatorname{Ann}_{\bZ}}
\newcommand{\Ansz}{\operatorname{Ann}_{{}^*\bZ}}
\newcommand{\Ker}{\operatorname{Ker}}
\newcommand{\T}{\mathrm{T}}

\begin{document}
%
%
%
%
%
\title[A uniform stability principle for dual lattices]%
{A uniform stability principle for dual lattices}

\author[M.~Vodi\v{c}ka and P.~Zlato\v{s}]%
{Martin Vodi\v{c}ka and Pavol Zlato\v{s}}

\address{%
{\sl Martin Vodi\v{c}ka and Pavol Zlato\v{s}}
\newline\indent
{\sl Faculty of Mathematics, Physics and Informatics}
\newline\indent
{\sl Comenius University}
\newline\indent
{\sl Mlynsk\'a dolina}
\newline\indent
{\sl 842\,48~Bratislava}
\newline\indent
{\sl Slovakia}
\newline\indent
{\rm vodicka6@uniba.sk}
\newline\indent
{\rm zlatos@fmph.uniba.sk}}



\thanks{Research of the second author supported by the grant
no.~1/0333/17 of the Slovak grant agency VEGA}

\begin{abstract}
We prove a highly uniform stability or ``almost-near'' theorem for dual lattices of
lattices $L \sbs \bR^n$. More precisely, we show that, for a vector $x$ from the linear
span of a lattice $L \sbs \bR^n$, subject to $\lambda_1(L) \ge \lambda > 0$, to be
$\eps$-close to some vector from the dual lattice $L'$ of \,$L$, it is enough that
the inner products $u\,x$ are $\delta$-close (with $\delta < 1/3$) to some integers for
all vectors $u \in L$ satisfying $\LV u \RV \le r$, where $r > 0$ depends on $n$,
$\lambda$, $\delta$ and $\eps$, only. This generalizes an analogous result proved for
integral lattices in \cite{MZ}. The proof is nonconstructive, using the ultraproduct
construction and a slight portion of nonstandard analysis.
\end{abstract}

\maketitle


\noindent
Informally, a property of objects of certain kind is ``stable'' if objects ``almost
satisfying'' this property are already ``close'' to objects having the property. For that
reason results establishing such a stability are frequently referred as a ``almost-near''
principles or theorems. Making precise the vague notions ``almost satisfying'' and
``close'' various rigorous notions of stability can be obtained. The study of stability
of functional equations originates from a question about the stability of additive
functions $\bR \to \bR$ and, more generally, of homomorphisms $G \to H$ between
metrizable topological groups, asked by Ulam, cf.~\cite{Mu}, \cite{U1}, \cite{U2}.
Since that time Ulam's type stability, modified in various ways, was studied for various
(systems of) functional equations\,---\,see, e.g., Rassias \cite{Ra}, Sz\'{e}kelyhidi
\cite{Sk}. A~systematic and general approach to this topic in the realm of compact
Hausdorff topological spaces, using nonstandard analysis was developed by Anderson
\cite{An}. The study of stability of the homomorphy property with respect to the
compact-open  topology was commenced by the second of the present authors \cite{Z1},
\cite{Z2}, \cite{Z3}. The survey article by Boualem and Brouzet \cite{BB} reflects some
recent development.

In the present paper we will prove the stability theorem for dual lattices stated in the
abstract, as well as some closely related results. Typically, such a stability result
would be formulated in a weaker form, namely that every vector $x$ from the linear span
of \,$L$, behaving almost like a vector from the dual lattice $L'$ of \,$L$ in the
sense that all its inner products $u\,x = u_1 x_1 + \ldots + u_n x_n$ with vectors $u$
from a ``sufficiently big'' subset of \,$L$ are ``sufficiently close'' to some integer,
is already ``arbitrarily close'' to a vector $y \in L'$. Below is the precise
formulation arising from this account. In it $\spn(L)$ denotes the linear subspace of
\,$\bR^n$ generated by $L$,
$\lv a\rv_\bZ = \min_{c \in \bZ}\lv a - c\rv =
   \min\bigl(a - \lfloor a \rfloor, \lceil a\rceil - a\bigr)$
denotes the distance of the real number $a$ from the set of all integers $\bZ$,
and $\LV x\RV = \sqrt{x\,x}$ is the euclidean norm, induced by the usual inner
(scalar) product $x\,y$ on $\bR^n$.

\begin{thm}\label{prelim}
Let $L \sbs \bR^n$ be a lattice. Then, for each $\eps > 0$, there exist
$\delta > 0$ and $r > 0$ such that for every $x \in \spn(L)$, satisfying
$\lv u\,x\rv_\bZ \le \delta$ for all $u \in L$, $\LV u \RV \le r$, there is
a $y \in L'$ such that $\LV x - y\RV \le \eps$.
\end{thm}

Such a statement naturally raises the question how the parameters $\delta$ and $r$
depend on the parameters $n$ and $\eps$ and some properties of the lattice $L$. We,
in fact, will prove a stronger and more uniform result, answering partly this question.
Namely, we will show that one can pick any $\delta \in (0,1/3)$; then $r$ can be chosen
depending on $n$, $\eps$, $\delta$ and, additionally, the Minkowski first successive
minimum $\lambda_1(L)$. The precise formulation is given in Theorem~\ref{stab-dual-latt-st}.
On the other hand, as the proof of this result uses the ultraproduct construction, it
establishes the mere existence of such an $r$, without any estimate of its size.

Theorem~\ref{stab-dual-latt-st} generalizes an analogous result proved in \cite{MZ} for
integral lattices, replacing the condition $L \sbs \bZ^n$ by introducing an additional
parameter $\lambda > 0$ and requiring $\lambda_1(L) \ge \lambda$. The mentioned result
in \cite{MZ} was obtained as a byproduct of a stability result for characters of
countable abelian groups the proof of which used Pontryagin-van\,Kampen duality between
discrete and compact groups and the ultraproduct construction. Our present result is
based on an intuitively appealing almost-near result (Theorem~\ref{stab-dual-latt-ns})
formulated in terms of nonstandard analysis which is linked to its standard counterpart
(Theorem~\ref{stab-dual-latt-st}) via the ultraproduct construction applied to a sequence
of lattices. As a consequence, Pontryagin-van\,Kampen duality is eliminated from the
proof. Additionally, the passage from stability of characters to stability of dual lattices
in \cite{MZ} naturally led to a formulation in terms of the pair of mutually dual norms
$\LV x\RV_1 = \lv x_1\rv + \ldots + \lv x_n\rv$ and
$\LV x\RV_\infty = \max(\lv x_1\rv,\dots,\lv x_n\rv)$. In our present work, starting
right away from lattices, the (equivalent) formulation in terms of the (selfdual)
euclidean norm $\LV x\RV = \LV x\RV_2$ seems more natural.

\section{Lattices and dual lattices}\label{1}

\noindent
We assume some basic knowledge of lattices or, more generally, of ``geometry of numbers''.
The readers can consult, e.g., Cassels \cite{C}, Gruber, Lekkerkerker \cite{GL} or
Lagarias \cite{L2}; however, for their convenience we list here the definitions of most
notions we use and some facts we build on.

A subgroup $L$ of the additive group $\bR^n$, where $n \ge 1$, is called a \textit{lattice}
if it is discrete, i.e., there is a $\lambda > 0$ such that $\LV x - y\RV \ge \lambda$ for
any distinct vectors $x,y \in L$. $\bR^n$ is alternatively viewed as a vector space or an
affine space and its elements as vectors or points, respectively. The dimension of the linear
space $\spn(X)$ generated by a set $X \sbs \bR^n$ is called the \textit{rank} of \,$X$, i.e.,
$\rank(X) = \dim\spn(X)$. A \textit{full rank lattice} is a lattice of rank equal the
dimension of the ambient space $\bR^n$. A \textit{body} is a nonempty bounded connected
set $C \sbs \bR^n$ which equals the closure of its interior. A body $C$ is called
\textit{centrally symmetric} if \,$-x \in C$ for any $x \in C$; it is called \textit{convex}
if \,$ax + (1-a)y \in C$ for any $x,y \in C$ and $a \in [0,1]$. An example of a centrally
symmetric convex body is the euclidean unit ball $B = \{x \in \bR^n: \LV x\RV \le 1\}$.
The \textit{Minkowski successive minima} of \,$L$ (with respect to the unit ball $B$) are
defined by
$$
\lambda_k(L) = \inf\{\lambda \in \bR: \lambda > 0,\ \rank(L \cap \lambda B) \ge k\}
$$
for $1 \le k \le \rank(L)$. In particular,
$\lambda_1(L) = \inf\{\LV x\RV: 0 \ne x \in L\}$. The \textit{covering radius} of \,$L$
is defined by
$$
\mu(L) = \inf\{r \in \bR: r > 0,\ \spn(L) \sbs L + r B\}\,.
$$
In all these cases the infima are in fact minima.

A \textit{basis} of a lattice $L \sbs \bR^n$ is an ordered $m$-tuple
$\bbeta = (v_1,\dots,v_m)$ of linearly independent vectors from $L$ which generate $L$
as a group, i.e.,
$$
L = \grp(v_1,\dots,v_m) = \{c_1 v_1 + \ldots + c_m v_m: c_1,\dots, c_m \in \bZ\}\,.
$$
Obviously, in such a case $\rank(L) = m$. In the proof of the fact that every lattice has
a basis the following elementary lemma, to which we will refer within short, plays a key
role.

\begin{lem}\label{basis-ext}
Let $L \sbs \bR^n$ be a lattice of rank $m$ and $(v_1,\dots,v_k)$, with $k < m$, be an
ordered $k$-tuple of linearly independent vectors from $L$ which can be extended to
a basis of \,$L$. Denote $V = \spn(v_1,\dots,v_k)$ and assume that the vector
$v_{k+1} \in L \sms V$ has a minimal (euclidean) distance to the linear subspace $V$
from among all the vectors in $L \sms V$. Then the $(k+1)$-tuple $(v_1,\dots,v_k,v_{k+1})$
either is already a basis of \,$L$ (if \,$k+1 = m$) or it can be extended to a basis of
\,$L$ (if \,$k + 1 < m$).
\end{lem}

Another useful consequence of the fact that every lattice has a basis is the following

\begin{lem}\label{integer-indep}
Let $L \sbs \bR^n$ be a lattice. Then a $k$-tuple of vectors $v_1,\dots,v_k \in L$ is
linearly independent if and only if, for any {\rm integers} $c_1,\dots,c_k \in \bZ$,
the equality $c_1 v_1 + \ldots + c_k v_k = 0$ implies $c_1 = \ldots = c_k = 0$.
\end{lem}

A basis $(v_1,\dots,v_m)$ of a lattice $L$ is \textit{Minkowski reduced} if, for each
$k \le m$, $v_k$ is the shortest vector from $L$ such that the $k$-tuple $(v_1,\dots,v_k)$
can be extended to a basis of \,$L$. It is known that every lattice has a Minkowski reduced
basis.

For any subset $S \sbs \bR^n$ we denote by
$$
\Anz(S) = \{x \in \bR^n: \all u \in S\: u\,x \in \bZ\}
$$
the \textit{integral annihilator} of \,$S$. Obviously, $\Anz(S)$ is a subgroup of \,$\bR^n$
for every $S \sbs \bR^n$, however, even for a lattice $L \sbs \bR^n$, the integral annihilator
$\Anz(L)$ need not be a lattice, unless $\rank(L) = n$. The \textit{dual lattice} of \,$L$
(also called the \textit{polar} or \textit{reciprocal lattice}) is defined as the intersection
$$
L' = \Anz(L) \cap \spn(L)\,.
$$
Then $L'$ is a lattice in $\bR^n$ of the same rank as $L$ and there is an obvious
duality relation $L'' = L$. The Minkowski successive minima of the original
lattice $L$ and its dual lattice $L'$ are related through a bound due to Banaszczyk
\cite{Bn}. Similarly, the covering radius of the dual lattice $L'$ can be estimated
in terms of the first Minkowski minimum of \,$L$\,---\,see Lagarias, Lenstra,
Schnorr~\cite{LLS}. Actually, in the quoted papers these results were stated and proved
for full rank lattices, i.e., in case $m = n$, only. However, introducing an orthonormal
basis in the linear subspace $\spn(L)$ and replacing any vector $x \in \spn(L)$ by its
coordinates with respect to it, they can be readily generalized as follows.

\begin{lem}\label{cov-rad}
Let $L \sbs \bR^n$ be a lattice of rank $m$. Then
$$
\lambda_k(L)\,\lambda_{m-k+1}\bigl(L'\bigr) \le m
$$
for each $k \le m$, and
$$
\lambda_1(L)\,\mu\bigl(L'\bigr) \le \frac{1}{2}\,m^{3/2}\,.
$$
\end{lem}

\section{Ultraproducts of lattices}\label{2}

\noindent
In order to keep our presentation self-contained, we give a brief account of the
ultraproduct construction and some notions of nonstandard analysis here. Nonetheless,
the readers are strongly advised to consult some more detailed exposition such as those
in Chang-Keisler \cite{CK}, Davis \cite{D} and Henson \cite{H}.

A nonempty system $D$ of subsets of a set $I$ is a called a \textit{filter} on $I$
if \,$\emptyset \nin D$, $D$ is closed with respect to intersections, and, for any
$X \in D$, $Y \sbs I$, the inclusion $X \sbs Y$ implies $Y \in D$. A filter $D$ on $I$ is
called an \textit{ultrafilter} if for any $X \sbs I$ either $X \in D$ or $I \sms X \in D$.
Ultrafilters of the form $D = \{X \sbs I: j \in X\}$, where $j \in I$, are called
\textit{principal}. As a consequence of the \textit{axiom of choice}, every filter on $I$
is contained in some ultrafilter; in particular, nonprincipal ultrafilters exist on every
infinite set $I$.

Given a set $I$ and a family of first order structures $(A_i)_{i \in I}$ of some first
order language $\Lam$, we can form their direct product $\prod_{i\in I} A_i$ with basic
operations and relations defined componentwise. If, additionally, $D$ is a filter on $I$,
then
$$
\alpha \equiv_D \beta \IFF \{i \in I:\ \alpha(i) = \beta(i)\} \in D
$$
defines an equivalence relation on $\prod A_i$. Denoting by $\alpha/D$ the coset of
a function $\alpha \in \prod A_i$ with respect to $\equiv_D$, the quotient
$$
B = \prod A_i \Big{/}\! D = \prod A_i \Big{/}\!\!\equiv_D\,,
$$
naturally becomes a $\Lam$-structure once we define
$$
f^B(\alpha_1/D,\dots,\alpha_p/D) = \beta/D\,,
$$
where $\beta(i) = f^{A_i}(\alpha_1(i),\dots,\alpha_p(i))$, for any $p$-ary
functional symbol $f$, and
$$
(\alpha_1/D,\dots,\alpha_p/D) \in R^B \IFF
\bigl\{i \in I: (\alpha_1(i),\dots,\alpha_p(i)) \in R^{A_i}\bigr\} \in D
$$
for any $p$-ary relational symbol $R$. Then $B$ is called the \textit{filtered} or
\textit{reduced product} of the family $(A_i)$ with respect to the filter $D$.
If \,$A_i = A$ is the same structure for each $i \in I$, then the reduced product
$$
A^I\!\big{/} D = \prod A_i \Big{/}\! D
$$
is called the \textit{filtered} or \textit{reduced power} of the $\Lam$-structure $A$.
If \,$D$ is an ultrafilter, then we speak of \textit{ultraproducts} and \textit{ultrapowers}.

The key property of ultraproducts is the following

\begin{lem}\label{Los}{\rm [{\L}os Theorem]}
Let $(A_i)_{i \in I}$ be a family of structures of some first order language $\Lam$,
$D$ be an ultrafilter on the index set $I$, $\Phi(x_1,\dots,x_p)$ be a $\Lam$-formula
and $\alpha_1,\dots,\alpha_p \in \prod A_i$. Then the statement
$\Phi(\alpha_1/D,\dots,\alpha_p/D)$ holds in the ultraproduct $\prod A_i\bigl{/}D$
if and only if \,$$
\bigl\{i \in I: \text{$\Phi(\alpha_1(i),\dots,\alpha_p(i))$ {\rm holds in} $A_i$}\bigr\}
\in D\,.
$$
\end{lem}

As a consequence, the canonical embedding of any $\Lam$-structure $A$ into its ultrapower
$\Ast = A^I\!\big{/}D$ is \textit{elementary}. More precisely, identifying every element
$a \in A$ with the coset $\bar a/D$ of the constant function $\bar a(i) = a$, we have
$$
\text{$\Phi(a_1,\dots,a_p)$\, holds in $A$ \ $\IFF$ \
$\Phi(a_1,\dots,a_p)$\, holds in $\Ast$}
$$
for every $\Lam$-formula $\Phi(x_1,\dots,x_n)$ and any $a_1,\dots,a_p \in A$. This
equivalence will be referred to as the \textit{transfer principle}.

The above accounts almost directly apply to many-sorted structures, like modules
over rings or vector spaces over fields, as well (see \cite{H}). In particular, if
\,$(V)_{i \in I}$ is a family of vector spaces over a field $F$, then the ultraproduct
$\prod V_i\big{/}D$ becomes a vector space over the ultrapower $ {}^{*\!}F = F^I\!\big{/}D$,
which is a field elementarily extending $F$. Similarly, if \,$(G_i)_{i \in I}$ is a family
of abelian groups, viewed as modules over the ring of integers $\bZ$, then the ultraproduct
$\prod G_i\bigl{/}D$ becomes not only an abelian group but also a module over the ring of
\textit{hyperintegers} $\sbZ = \bZ^I\!\bigl{/}D$, elementarily extending the ring $\bZ$.
And, what is of crucial importance, the {\L}os Theorem is still true for formulas in
the corresponding two-sorted language.

From now on $I = \{1,2,3,\dots\}$ denotes the set of all positive integers and $D$ is some
fixed nonprincipal ultrafilter on $I$. We form the ordered field
of \textit{hyperreal numbers} as the ultrapower $\sbR = \bR^I\!\big{/}D$ of the ordered
field $\bR$. Then
\begin{align*}
\FsR &= \{x \in \sbR: \exs r \in \bR,\, r > 0 \: \lv x\rv < r\}\,,\\
\IsR &= \{x \in \sbR: \all r \in \bR,\, r > 0 \: \lv x \rv < r\}
\end{align*}
denote the sets of all \textit{finite} hyperreals and of all \textit{infinitesimals},
respectively. It can be easily verified that $\FsR$ is a subring of \,$\sbR$ and $\IsR$
is an ideal in $\FsR$. Hyperreal numbers not belonging to $\FsR$ are called
\textit{infinite}. For $x \in \sbR$ we sometimes write $\lv x\rv < \infty$ instead
of \,$x \in \FsR$, and $x \sim \infty$ instead of \,$x \nin \FsR$. Two hyperreals $x$, $y$
are said to be \textit{infinitesimally close}, in notation $x \apr y$, if
\,$x - y \in \IsR$. Moreover, for each $x \in \FsR$, there is a unique real number
$\xo \in \bR$, called the \textit{standard part} of \,$x$, such that $x \apr \xo$. As a
consequence, $\FsR/\IsR \iso \bR$ as ordered fields.

A hyperreal number $x = \alpha/D$, where $\alpha\:I \to \bR$, is finite if and only
if there is a positive $r \in \bR$ such that $\{i \in I: \lv \alpha(i)\rv < r\} \in D$;
this is equivalent to the convergence of the sequence $\alpha$ to $\xo$ with respect to
the filter $D$. In particular, $x$ is infinitesimal if and only if \,$\xo = 0$, i.e., if
and only if the sequence $\alpha$ converges to 0 with respect to~$D$. As $D$ necessarily
extends the Frechet filter, $\lim_{i \to \infty} \alpha(i) = a \in \bR$ in the usual sense
implies $\xo = a$, i.e., $x \apr a$.

The standard part map has the following homomorphy properties with respect to the field
operations:
$$
{}^\co(x + y) = \xo + \yo
\quad\text{and}\quad
{}^\co(x\,y) = \xo\,\yo
$$
for any $x,y \in \FsR$, and if additionally $x \napr 0$, then also
${}^\co\bigl(x^{-1}\bigr) = \bigl(\xo\bigr)^{-1}$.

Along with the equivalence relation of infinitesimal nearness $\apr$, we introduce the
relation of \textit{archimedean equivalence} $\sim$ or \textit{order equality} on $\sbR$
as follows:
$$
x \sim 0 \Iff x = 0\,,
\quad \text{and}\quad
x \sim y \Iff 0 \napr \lv\frac{x}{y}\rv < \infty \ \text{\,for \,$x,y \ne 0$}\,.
$$
When $x \sim y$ we say that $x$ and $y$ are of the \textit{same (archimedean) order}.
We also say that $x$ is of \textit{smaller order} than $y$ or that $y$ is of
\textit{bigger order} than $x$, in symbols $x \ll y$, if \,$y \ne 0$ and
$\frac{x}{y} \apr 0$. Obviously, for $x,y \ne 0$, $x \sim y$ is equivalent to
neither $x \ll y$ nor $y \ll x$.

According to the transfer principle, we can identify, for any finite intger $n \ge 1$,
the vector space $(\sbR)^n$ over the field $\sbR$ and the ultrapower
${}^*(\bR^n) = (\bR^n)^I\!\big{/}D$, so that the notation $\sbR^n$ is unambiguous.
More generally, for any subset $S \sbs \bR^n$ we identify the ultrapower
${}^*S = S^I\!\big{/}D$ with the subset
$$
\bigl\{(\alpha_1/D,\dots,\alpha_n/D) \in \sbR^n:
   \{i \in I: (\alpha_1(i),\dots,\alpha_n(i)) \in S\} \in D\bigr\}
$$
of \,$\sbR^n$. The inner product on $\bR^n$ extends to the inner product on $\sbR^n$,
preserving all its first order properties. In order to distinguish the linear spans
with respect to the fields $\bR$ and $\sbR$, respectively, we introduce the
\textit{internal linear span} of a set $X \sbs \sbR^n$ which, due to the fact that
the ambient vector space $\sbR^n$ has finite internal dimension $n$, can be described
as follows:
$$
\sspn(X) = \{a_1x_1 + \ldots + a_nx_n: x_1,\dots,x_n \in X,\ a_1,\dots,a_n \in \sbR\}\,,
$$
We also distinguish the lattice or subgroup, i.e., the $\bZ$-submodule $\grp(v_1,\dots,v_m)$
of \,$\bR^n$ generated by vectors $v_1,\dots,v_m \in \bR^n$, and the internal lattice
\textit{internally generated} by vectors $v_1,\dots,v_m \in \sbR^n$, i.e., the
$\sbZ$-submodule
$$
\sgrp(v_1,\dots,v_m) = \{c_1 v_1 + \ldots + c_m v_m: c_1,\dots,c_m \in \sbZ\}
$$
of \,$\sbR^n$.

Similarly as in $\sbR$, vectors from $\FsR^n$ are called \textit{finite} and vectors
from $\IsR^n$ are called \textit{infinitesimal}. Obviously,
\begin{align*}
\FsR^n &= \bigl\{x \in \sbR^n: \LV x\RV < \infty\bigr\}\,,\\
\IsR^n &= \bigl\{x \in \sbR^n: \LV x\RV \apr 0\bigr\}\,.
\end{align*}
Both $\FsR^n$ and $\IsR^n$ are vector spaces over $\bR$ and even modules over $\FsR$, but
not over $\sbR$. Vectors $x, y \in \sbR^n$ are said to be \textit{infinitesimally close},
in notation $x \apr y$, if \,$x - y \in \IsR^n$, i.e., if \,$\LV x - y\RV \apr 0$. The
\textit{standard part} of a vector $x = (x_1,\dots,x_n) \in \FsR^n$ is the vector
$\xo = (\xo_1,\dots,\xo_n)$; obviously, $\xo$ is the unique vector in $\bR^n$
infinitesimally close to $x$. Then $\FsR^n\!/\IsR^n \iso \bR^n$ as vector spaces
over $\bR$.

Though the ultraproduct construction can be applied to any family of lattices
$L_i \sbs \bR^{n_i}$, it is sufficient for our purpose to deal with lattices
situated in the same ambient vector space $\bR^n$ with $n \ge 1$ fixed. Given a
sequence $(L_i)_{i \in I}$ of lattices $L_i \sbs \bR^n$ we can form the ultraproduct
$\prod L_i\big{/}D$ and identify it with the subset
$$
L = \bigl\{(\alpha_1/D,\dots,\alpha_n/D) \in \sbR^n:
         \{i \in I: (\alpha_1(i),\dots,\alpha_n(i)) \in L_i\} \in D\bigr\}
$$
of the vector space $\sbR^n$ over $\sbR$. Then $L$ is an internal discrete additive
subgroup of \,$\sbR^n$, i.e., it is a module over the ring of \textit{hyperintegers}
$\sbZ$ and there is a positive $\lambda \in \sbR$ such that $\LV x - y\RV \ge \lambda$
for any distinct $x, y \in L$; however, it should be noticed that $\lambda$ may well be
infinitesimal. Moreover, as $D$ is an ultrafilter, there is an $m \le n$ and a set
$J \in D$ such that $\rank(L_i) = m$ for each $i \in J$. We write $\rank(L) = m$ and refer
to $L$ as an \textit{internal lattice} in $\sbR^n$ of rank $m$. Then we can assume, without
loss of generality, that $\rank(L_i) = m$ for each $i \in I$. The Minkowski successive
minima of such an internal lattice $L$ can be defined in two ways which are equivalent
by the transfer principle:
$$
\lambda_k(L) = \bigl(\lambda_k(L_i)\bigr)_{i\in I}\big{/}D
= \min\{\lambda \in \sbR: \lambda > 0,\ \rank(L \cap \lambda \Bst) \ge k\}
$$
for $k \le m$. Then $0 < \lambda_1(L) \le \ldots \le \lambda_m(L)$ is a sequence
of \textit{hyperreal} numbers, hence it can contain both infinitesimals as well as
infinite hyperreals. Additionally, we put
\begin{align*}
\rank_0(L) &= \#\{k: 1 \le k \le m,\ \lambda_k(L) \apr 0\}\,,\\
\rankf(L)  &= \#\{k: 1 \le k \le m,\ \lambda_k(L) < \infty\}\,,
\end{align*}
where $\#\,H$ denotes the number of elements of a finite set $H$. Note that
$\rank_0(L) = 0$ if \,$\lambda_1(L) \napr 0$, as well as $\rankf(L) = 0$ if
\,$\lambda_1(L) \nin \FsR$. Obviously, if \,$\rank_0(L) > 0$, then it is the biggest
$k \le m$ such that $\lambda_k(L) \apr 0$; similarly, if \,$\rankf(L) > 0$, then it
is the biggest $k \le m$ such that $\lambda_k(L) < \infty$.

At the same time, we can assume that $\beta_1,\dots,\beta_m \in \prod L_i$ are functions
such that, for each $i \in I$ (or at least for each $i$ from some set $J \in D$), the
$m$-tuple of vectors $\bbeta(i) = (\beta_1(i),\dots,\beta_m(i))$ is a Minkowski reduced
basis of the lattice $L_i$. Then, due to {\L}os Theorem (Lemma~\ref{Los}), the $m$-tuple
$\bbeta/D = (v_1,\dots,v_m)$, where $v_k = \beta_k/D$ for $k \le m$, is a Minkowski
reduced basis of the internal lattice $L$, i.e., the vectors $v_1,\dots,v_m$ are linearly
independent over $\sbR$ and generate $L$ as a $\sbZ$-module.

\begin{lem}\label{mink-bas-min}
Let $L \sbs \sbR^n$ be an internal lattice of rank $m$ and $\bbeta = (v_1,\dots,v_m)$ be a
Minkowski reduced basis of \,$L$. Then the following hold true:
\begin{enum}
\item[\sl (a)]
If \,$\LV v_k\RV \ll \LV v_{k+1}\RV$ for some $k < m$ and \,$V = \sspn\{v_1,\dots,v_k\}$,
then \newline
$\LV x\RV \nll \LV v_{k+1}\RV$ for every vector $x \in L \sms V$.
\item[\sl (b)]
$\LV v_k\RV \sim \lambda_k(L)$ for each $k \le m$.
\end{enum}
\end{lem}

\begin{proof}
(a) Assume that, under the assumptions of (a), we have
$\LV x\RV \ll \LV v_{k+1}\RV$ for some $x \in L \sms V$. We denote the orthogonal
projection of a vector $y \in \sbR^n$ to $V$ by $y_V$. Let $z \in L \sms V$ be a
vector such that its the distance $z - z_V$ to $V$ is minimal from among all the
vectors $y \in L \sms V$. Therefore,
$$
\LV z - z_V\RV \le \LV x - x_V\RV \le \LV x\RV\,.
$$
As $z_V \in V$, there are hyperreals $a_1,\dots,a_k \in \sbR$ such that
$z_V = a_1 v_1 + \ldots + a_k v_k$. Denoting $c_j = \lfloor a_j\rfloor$ their lower
integer parts and $z' = z - c_1 v_1 - \ldots - c_k v_k \in L$, we have $z - z' \in V$,
hence $\LV z' - z'_V\RV = \LV z - z_V\RV$, so that the vector $z' \in L$ has the same
minimality property as $z$. Then, according to Lemma~\ref{basis-ext} and the transfer
principle, the $(k+1)$-tuple $(v_1,\dots,v_k,z')$ can be extended to a basis of \,$L$,
hence $\LV v_{k+1}\RV \le \LV z'\RV$, as the basis $(v_1,\dots,v_m)$ is Minkowski
reduced. At the same time,
$$
z'_V =  (a_1 - c_1)v_1 + \ldots + (a_k - c_k)v_k\,,
$$
with $\lv a_j - c_j\rv < 1$ for each $j \le k$. From the triangle inequality  we get
\begin{align*}
\LV z'\RV &\le \LV z'_V\RV + \LV z' - z'_V\RV \\
   &= \LV (a_1 - c_1)v_1 + \ldots + (a_k - c_k)v_k\RV + \LV z - z_V\RV \\
   &< \LV v_1\RV + \ldots + \LV v_k\RV + \LV x\RV\,.
\end{align*}
Therefore, $\LV z'\RV \ll \LV v_{k+1}\RV$, hence $\LV z'\RV < \LV v_{k+1}\RV$,
which is a contradiction.

(b) Because $\LV v_1\RV = \lambda_1(L)$, the statement of (b) is true for $k=1$.
Assume, toward a contradiction, that $k < m$ for the biggest index satisfying
$\LV v_k\RV \sim \lambda_k(L)$. Then
$$
1 \le \frac{\LV v_k\RV}{\lambda_k(L)} < \infty
\qquad\text{and}\qquad
\frac{\lambda_{k+1}(L)}{\LV v_{k+1}\RV} \apr 0\,.
$$
Therefore,
$$
\frac{\LV v_k\RV}{\LV v_{k+1}\RV} \le
\frac{\lambda_{k+1}(L)}{\lambda_k(L)} \cd \frac{\LV v_k\RV}{\LV v_{k+1}\RV}
= \frac{\LV v_k\RV}{\lambda_k(L)} \cd \frac{\lambda_{k+1}(L)}{\LV v_{k+1}\RV} \apr 0\,.
$$
Then, according to (a), $\frac{\LV x\RV}{\LV v_{k+1}\RV} \napr 0$ for every vector
$x \in L \sms \sspn(v_1,\dots,v_k)$. In particular,
$\frac{\lambda_{k+1}(L)}{\LV v_{k+1}\RV} \napr 0$.
\end{proof}

\begin{rem}\label{rem1}
(b) of Lemma~\ref{mink-bas-min} follows immediately, by applying the transfer principle,
from the following estimates of the lengths of vectors in any Minkowski reduced basis
$(v_1,\ldots,v_m)$ of a rank $m$ lattice $L \sbs \bR^n$ in terms of its Minkowski successive
minima:
$$
\lambda_k(L) \le \LV v_k\RV \le 2^k \lambda_k(L)
$$
for all $k \le m$ (see Lagarias \cite{L1}; Mahler \cite{Mh1} has even better upper bounds).
Then (a) could be proved as an easy consequence of (b). However, it is perhaps worthwhile
to notice that, using the internal lattice concept, the purely qualitative estimates (a),
(b) follow already from Lemma~\ref{basis-ext} and the existence of Minkowski reduced
bases.
\end{rem}

The \textit{standard part} ${}^{\co\!}X$ of a set $X \sbs \sbR^n$ consists of the standard
parts of all finite vectors from $X$; alternatively, it can be formed by taking the quotient
of the set of finite vectors from $X$ with respect to the equivalence relation of infinitesimal
nearness. Identifying the results of both approaches, we have
$$
{}^{\co\!}X =\bigl(X \cap \FsR^n\bigr)/\!\!\apr\
            = \bigl\{\xo: x \in X \cap \FsR^n\bigr\}
    = \bigl\{y \in \bR^n: \exs x \in X: y \apr x\bigr\}\,.
$$
In particular, for an additive subgroup $G \sbs \sbR^n$ we denote
$$
\FG = G \cap \FsR^n
\qquad\text{and}\qquad
\IG = G \cap \IsR^n
$$
the additive subgroups of \,$\sbR^n$ formed by the finite and infinitesimal elements in $G$,
respectively. Then its standard part ${}^\co G$ is an additive subgroup of \,$\bR^n$ which
can be identified with the quotient
$$
{}^\co G = \FG/\IG\,.
$$
However, even for an internal lattice $L \sbs \sbR^n$, its standard part $\Lo$ is not
necessarily discrete, hence it need not be a lattice in $\bR^n$. A more detailed account
will follow after a preliminary lemma.

\begin{lem}\label{observe-indep}
Let $L \sbs \sbR^n$ be an internal lattice of rank $m$ and $\bbeta = (v_1,\dots,v_m)$ be a
Minkowski reduced basis of \,$L$ such that all the vectors in $\bbeta$ are infinitesimal.
Then there exist hyperintegers $c_1,\dots,c_m \in \sbZ$ such that all the vectors
$c_k v_k$ are finite but not infinitesimal and $c_k$ divides $c_{k-1}$ whenever
$2 \le k \le m$. For such a choice of \,$c_1$, \dots, $c_m$ the internal sublattice
$M = \sgrp(c_1v_1,\dots,c_mv_m) \sbs L$ contains no infinitesimal vector except for~$0$,
in other words $\lambda_1(M) \napr 0$.
\end{lem}

\begin{proof}
Let's start with an arbitrary $c_m \in \sbZ$ such that $c_m v_m \in \FL \sms \IL$
(e.g., one can put $c_m = \Bigl\lceil\LV v_m\RV^{-1}\Bigr\rceil$ guaranteeing that
$1 \le \LV c_m v_m\RV < 1 + \LV v_m\RV \apr 1$). Further we proceed by backward recursion.
Assuming that $2 \le k \le m$ and $c_k$ is already defined, we put $c_{k-1} = c_k$ if
\,$c_k v_{k-1} \napr 0$ (as $\LV v_{k-1}\RV \le \LV v_k\RV$, $c_k v_{k-1} \in \FL$ is
satisfied automatically), otherwise we put $c_{k-1} = b c_k$ where $b \in \sbZ$ is
any hyperinteger such that $b c_k v_{k-1} \in \FL \sms \IL$ (e.g.,
$b = \Bigl\lceil\LV c_kv_{k-1}\RV^{-1}\Bigr\rceil$ will work). Obviously,
$c_k \in \sbZ$ divides $c_{k-1} \in \sbZ$ for any $2 \le k \le m$.

Assume that $x \apr 0$ where $x = a_1c_1v_1 + \ldots + a_mc_mv_m$ for some
$a_1,\dots,a_m \in \sbZ$, not all equal to~0. Let $q \le m$ be the biggest index
such that $a_q \ne 0$. Then
$$
x' = \frac{1}{c_q}\,x = \sum_{k=1}^q \frac{a_k c_k}{c_q}\,v_k \ne 0
$$
is a vector from the internal lattice $L$. Moreover, $c_q x' = x \apr 0$, while
\hbox{$c_q v_q \napr 0$,} hence $\LV x'\RV \ll \LV v_q\RV$. Let $p \le q$ be the
smallest index such that $\LV x'\RV \ll \LV v_p\RV$. Denote
$\lambda = \LV x'\RV$ if \,$p = 1$, or $\lambda = \max(\LV v_{p-1}\RV, \LV x'\RV)$ if
\,$p > 1$. Then the hyperball $\lambda\,\Bst$ contains $p$ linearly independent vectors
$v_1,\dots,v_{p-1},x'$ from $L$, hence $\lambda_p(L) \le \lambda$ and, at the same time,
$\lambda \ll \LV v_p\RV$, contradicting Lemma~\ref{mink-bas-min}\,(b).
\end{proof}

\begin{prop}\label{stand-part}
Let $L = \prod L_i\big{/}D \sbs \sbR^n$ be an internal lattice of rank $m$ and $\Lo$ be
its standard part. Then the following hold true:
\begin{enum}
\item[\sl (a)]
$\Lo$ is a lattice in $\bR^n$ if and only if there is a positive $\lambda \in \bR$
such that the set $\{i \in I: \lambda_1(L_i) \ge \lambda\}$ belongs to $D$. This
is equivalent to $\lambda_1(L) \napr 0$ as well as to $\rank_0(L) = 0$.
\item[\sl (b)]
$\Lo$ is the direct sum of a linear subspace of \,$\bR^n$ of dimension $\rank_0(L)$
and a lattice in $\bR^n$ of rank $\rankf(L) - \rank_0(L)$.
\item[\sl (c)]
$\Lo$ is a lattice of rank $q \le m$ if and only if \,$\rank_0(L) = 0$ and
\,$\rankf(L) = q$.
\end{enum}
\end{prop}

\begin{proof}
(a) The equivalence of any of the first two conditions to the discreteness of the
group $\Lo$ is obvious. Similarly, any of the obviously equivalent conditions
$\lambda_1(L) \napr 0$ and $\rank_0(L) = 0$ implies the discreteness of \,$\Lo$.
Otherwise, there is at least one nonzero infinitesimal vector $v \in L$. Then one
can find a  hyperinteger $c \in \sbZ$ such that $cv$ is finite but not infinitesimal.
Obviously, its standard part $w = {}^\co(cv) \ne 0$ belongs to $\Lo$, so that
$\spn(w) = \bR w$ is a line in $\bR^n$. We prove the inclusion $\bR w \sbs \Lo$.
Taking any $x = aw \in \bR w$, with $a \in \bR$, and putting
$b = \lfloor ac\rfloor \in \sbZ$, we have $b \le ac < b + 1$ which, by the virtue
of \,$v \apr 0$, implies $bv \apr acv$. Hence
$$
x = aw \apr acv \apr bv \in \FL\,,
$$
and $x = {}^\co(bv) \in \Lo$. It follows that $\Lo$, containing the line
$\bR w \sbs \bR^n$, is not discrete.

(b) Let $(v_1,\dots,v_m)$ be a Minkowski reduced basis of \,$L$. Denote $p = \rank_0(L)$
and $q = \rankf(L)$. According to Lemma~\ref{mink-bas-min}\,(b), a vector $v_k$ is
infinitesimal if and only if \,$k \le p$, and it is finite if and only if \,$k \le q$.
For the same reason, if \,$x \in L \sms \sgrp(v_1,\dots,v_q)$ then $\LV x\RV \nll v_{q+1}$,
hence $x \nin \FL$. Therefore the standard part $\Lo$ of the internal lattice $L$
coincides with the standard part of its internal sublattice $\sgrp(v_1,\dots,v_q)$.
Due to Lemma~\ref{observe-indep}, there are hyperintegers $c_1,\dots,c_p\in \sbZ$ such
that $c_kv_k \in \FL \sms \IL$ for any $k$ and $c_k$ divides $c_{k-1}$ for $k \ge 2$.
Then the internal sublattice $M = \sgrp(c_1v_1,\dots,c_pv_p) \sbs L$ contains no nonzero
infinitesimal vector. Let us denote $w_k = {}^\co(c_kv_k)$ for $k \le p$, and, additionally,
$c_k = 1$, $w_k = {}^\co v_k = {}^\co(c_k v_k)$ for $p < k \le q$. As a consequence, $\Lo$
coincides with the sum of the linear subspace $\spn(w_1,\dots,w_p)$ and the lattice
$\grp(w_{p+1},\dots,w_q)$.

The proof of (b) will be complete once we establish the following claim.
\smallskip

\noindent
\textbf{Claim.}
\textit{The vectors $w_1,\dots,w_q$ are linearly independent over $\bR$.}
\smallskip

Indeed, let $b \in \sbN$ be any infinite hypernatural number. Put $v_k' = b^{-1} v_k$,
$c_k' = b c_k$ for any $k \le q$. Then all the vectors $v_1',\dots,v_q'$ are infinitesimal
and form a Minkowski reduced basis of the lattice $L' = \bigl\{b^{-1} x: x \in L\bigr\}$.
Now, all the vectors $c_k'v_k' = c_k v_k$, where $k \le q$, are finite but not infinitesimal
and $c_k'$ divides $c_{k-1}'$ for $k \ge 2$. From Lemma~\ref{observe-indep} we infer that the
internal lattice
$$
N = \sgrp(c_1v_1,\dots,c_qv_q) = \sgrp\bigl(c_1'v_1',\dots,c_q'v_q'\bigr)
$$
satisfies $\lambda_1(N) \napr 0$. Then, by (a), its standard part ${}^{\co\!}N$ is a
lattice in $\bR^n$. According to Lemma~\ref{integer-indep}, it suffices to show that
$a_1w_1 + \ldots + a_q w_q = 0$ implies $a_1 = \ldots = a_q = 0$ for any \textit{integers}
$a_1,\dots,a_q \in \bZ$. Since the first equality is equivalent to
$a_1 c_1 v_1 + \dots + a_q c_q v_q \apr 0$ and the left hand vector belongs to $N$,
which contains no infinitesimal vector except for $0$, we have
$a_1 c_1 v_1 + \dots + a_q c_q v_q = 0$,  and the desired conclusion follows
from the linear independence of the vectors $c_1 v_1,\dots,c_q v_q$ over $\sbR$.
\smallskip

(c) follows directly from (a) and (b). 
\end{proof}

Let us record the following direct consequence of (b).

\begin{cor}
Let $L$ be an internal lattice in $\sbR^n$. Then its standard part $\Lo$ is a closed
subgroup of the additive group $\bR^n$.
\end{cor}

\section{An ``almost-near'' result for systems of linear equations}\label{3}

\noindent
We denote by $F^{m \cx n}$ the vector space of all $m \cx n$ matrices over a field $F$.
Unless otherwise said, the vector space $F^n$ consists of column vectors. The transpose
of a matrix $A$ is denoted by $A^\T$. A matrix $A \in \sbR^{m \cx n}$ is called
\textit{finite}, in symbols $A \in \FsR^{m \cx n}$, if all its entries $a_{ij}$ are
finite. Then the matrix $\Ao = \bigl(\ao_{ij}\bigr) \in \bR^{m \cx n}$ is
called the \textit{standard part} of \,$A$. The preservation of addition and
multiplication by the standard part map on $\FsR$ extends to finite matrices, i.e.,
$$
{}^\co(A + B) = \Ao + \Bo
\qquad\text{and}\qquad
{}^\co(A\,C) = \Ao\,\Co
$$
for any $A, B \in \FsR^{m \cx n}$, $C \in \FsR^{n \cx p}$.

The following ``almost-near'' result for solutions of systems of linear equations
will be used in the proof of our first stability Theorem~\ref{stand-part-annih} in
the next section.

\begin{prop}\label{stab-linalg}
Let $A \in \FsR^{m \cx n}$ be any matrix such that its rows are linearly independent over
the field $\sbR$, the standard parts of its rows are linearly independent over \,$\bR$,
and $b \in \FsR^m$. Then, for any $x \in \FsR^n$ satisfying $A\,x \apr b$,
there is a $y \in \sbR^n$ such that $y \apr x$ and $A\,y = b$.
\end{prop}

Notice that the vector $y$, being infinitesimally close to the standard vector $x$,
is necessarily finite.

\begin{proof}
The above assumptions guarantee that $m \le n$ and both the systems $A\,\xi = b$,
$\Ao\,\xi = \bo$ indeed have solutions (in $\sbR^n$, $\bR^n$, respectively),
because the internal rank of \,$A$ over $\sbR$, as well as the rank of \,$\Ao$ over $\bR$
are both equal to $m$. We denote by $V$ the orthocomplement of the internal linear
subspace $\{\xi \in \sbR^n: A\,\xi = 0\}$ in $\sbR^n$.

Let $x \in \FsR^n$ satisfy $A\,x \apr b$ and $y \in \sbR^n$ be the orthogonal projection
of \,$x$ to the affine subspace $\{\xi \in \sbR^n: A\,\xi = b\}$ of \,$\sbR^n$. Then
$x - y \in V$ and, of course, $A\,y = b$. It suffices to prove that $x \apr y$.

Let $A = P\,D\,Q^\T$ be the singular value decomposition of $A$. Thus $P \in \sbR^{m \cx m}$,
$Q \in \sbR^{n \cx n}$ are orthogonal matrices and $D$ is a diagonal matrix with the
diagonal formed by the singular values $d_1 \ge \ldots \ge d_m > 0$ of \,$A$. Then
$\Ao = \Po\,\Do\,\Qo^\T$ is the singular value decomposition of \,$\Ao$, and from the
properties of \,$A$ it follows that all the singular values $\ds_1,\dots,\ds_m$ of \,$\Ao$
are still positive, hence all the $d_i$s are noninfinitesimal. The internal linear subspace
$V \sbs \sbR^n$ is spanned by the first $m$ columns of the matrix $Q$, and
$$
d_m \LV v\RV \le \LV A\,v\RV \le d_1 \LV v\RV\,,
$$
holds for each vector $v \in V$
(see, e.g., Han, Neumann \cite{HN}, \S\,5.6, and Bernstein \cite{Bn}, \S\S\,5.6, 9.11).
In particular, since $A\,x \apr b = A\,y$,
$$
d_m \LV x - y\RV \le \LV A(x - y)\RV \apr  0\,,
$$
implying $\LV x - y\RV \apr 0$, i.e., $x \apr y$.
\end{proof}

\section{The ``almost-near'' theorems for dual lattices \\
nonstandard formulation}\label{4}

\noindent
Given an internal lattice $L = \prod L_i\big{/}D$ in $\sbR^n$, its \textit{internal
integral annihilator} can be defined as the ultraproduct of the integral annihilators
of the particular lattices $L_i \sbs \bR^n$ or, equivalently, as the annihilator of
\,$L$ with respect to the set of hyperintegers $\sbZ$. Then the {\L}os Theorem
(Lemma~\ref{Los}) assures that both the objects coincide, i.e.,
$$
\Ansz(L) = \bigl\{u \in \sbR^n: \all x \in L\: ux \in \sbZ\bigr\}
         = \prod \Anz(L_i)\Big{/}D\,.
$$
Similarly, we have a two-fold definition of the \textit{internal dual} of the internal
lattice $L$:
$$
L' = \Ansz(L) \cap \sspn(L) = \prod L_i'\Big{/}D\,.
$$

Using the transfer principle, Lemma~\ref{cov-rad} implies the following transference
relations between the successive minima of an internal lattice $L \sbs \sbR^n$ and the
successive minima and the covering radius, respectively, of its internal dual lattice.

\begin{lem}\label{cov-rad-qual}
Let $L \sbs \bR^n$ be an internal lattice of rank $m$. Then
$$
\lambda_k(L)\,\lambda_{m-k+1}\bigl(L'\bigr) < \infty
$$
for each $k \le m$, and
$$
\lambda_1(L)\,\mu\bigl(L'\bigr) < \infty\,.
$$
\end{lem}

\begin{rem}\label{rem2}
The preceding relations follow already from the considerably weaker estimates than those
in Lemma~\ref{cov-rad}, namely,
$$
\lambda_k(L)\,\lambda_{m-k+1}\bigl(L'\bigr) \le m!\,,
$$
due to Mahler \cite{Mh2}, and the almost obvious observation
$$
\mu\bigl(L'\bigr) \le \frac{1}{2}\,m\,\lambda_m\bigl(L'\bigr)\,,
$$
which jointly imply
$$
\lambda_1(L)\,\mu\bigl(L'\bigr) \le \frac{1}{2}\,m\,m!\,.
$$
Yet weaker estimates $\lambda_k(L)\,\lambda_{m-k+1}\bigl(L'\bigr) \le (m!)^2$ (see
Gruber-Lekkerkerker \cite{GL}, p.~125) are still sufficient (cf.~Remark~\ref{rem1}).
\end{rem}

As first we prove an infinitesimal version of the ``almost-near'' result for integral
annihilators of internal lattices.

\begin{thm}\label{stand-part-annih}
Let $L \sbs \sbR^n$ be an internal lattice. Then for each $x \in \FsR^n$, such that
$\lv u\,x\rv_\bZ \apr 0$ for every finite $u \in L$, there is a $y \in \Ansz(L)$
such that $y \apr x$.
\end{thm}

\begin{proof}
Let $\bbeta = (v_1,\dots,v_m)$ be a Minkowski reduced basis of \,$L$, and
$0 \le p \le q \le m$ be natural numbers such that $v_1, \dots, v_p$ are all the
infinitesimal vectors in $\bbeta$ and $v_1,\dots,v_q$ are all the finite vectors
in $\bbeta$. Recalling Proposition~\ref{stand-part}\,(b) and its proof, there are
hyperintegers $c_1,\dots, c_p \in \sbZ$ such that the vectors
$c_1 v_1, \dots, c_p v_p \in L$ are finite and noninfinitesimal and each finite
vector $u \in L$ is infinitesimally close to a vector of the form
$$
(a_1 c_1 v_1 + \ldots + a_p c_p v_p) + (a_{p+1} v_{p+1} + \ldots + a_q v_q)\,,
$$
where $a_1,\dots,a_p \in \bR$ and $a_{p+1}, \dots, a_q \in \bZ$.

Let us form the matrix with columns $c_1 v_1,\dots,c_p v_p, v_{p+1}, \dots, v_q$, and
denote by $A \in \FsR^{q \cx n}$ its transpose. Then an $x \in \FsR^n$ satisfies the
condition $\lv u\,x\rv_\bZ \apr 0$ for each finite $u \in L$ if and only if
\,$c_k v_k\,x \apr 0$ for $k \le p$, and $\lv v_k\,x\rv_\bZ \apr 0$ for $p < k \le q$.
Assume that $u$ satisfies this condition and put
$b = \bigl(0,\dots,0, {}^\co(v_{p+1} x), \dots, {}^\co(v_q x)\bigr)^\T$.
Then $b \in \bZ^q$ and $x$ satisfies $A\,x \apr b$. By the virtue of
Proposition~\ref{stab-linalg}, there is a $y \in \FsR^n$ such that $y \apr x$ and
$A\,y = b$. Then, however, $v_k\,y = b_k = 0$ for $k \le p$, and $v_k\,y = b_k \in \bZ$
for $p < k \le q$. If \,$q = m$, we are done. Otherwise there exists a sequence of
integers $q = q_0 < q_1 < \ldots < q_t = m$ such that
$$
\bigl\|v_{q_{s-1}}\bigr\| \ll \bigl\|v_k\bigr\| \sim \bigl\|v_{q_s}\bigr\|
$$
for all $s$, $k$ satisfying $1 \le s \le t$, $q_{s-1} < k \le q_s$.

We are going to construct a sequence of vectors
$y^{(0)} = y,\, y^{(1)}, \dots, y^{(t)} \in \FsR^n$, such that $y^{(s)} \apr x$ and
$v_k\,y^{(s)} \in \sbZ$ for any $s \le t$, $k \le q_s$. Then already $v\,y^{(t)} \in \sbZ$
for every $v \in L$, as required. This will be achieved by an inductive argument.
Obviously, to this end it is enough to prove the following
\smallskip

\noindent
\textbf{Claim.}
\textit{Let \,$0 \le s < t$ and $z \in \FsR^n$ be a vector such that $v_k\,z \in \sbZ$
for any $k \le q_s$. Then there is a $z' \in \FsR^n$ such that $z' \apr z$ and
$v_k\,z' \in \sbZ$ for any $k \le q_{s+1}$.}
\smallskip

Let us denote $q' = q_s$, $q'' = q_{s+1}$, $d = q'' - q' > 0$, for typographical reasons,
and form the internal lattice $M = \sgrp(v_1,\dots,v_{q''}) \sbs L$, as well as the
internal linear subspace $V = \sspn(M) = \sspn(v_1,\dots,v_{q''}) \sbs \sbR^n$. According
to Lemma~\ref{mink-bas-min}\,(b) and Lemma~\ref{cov-rad-qual} we know that
$$
\LV v_k\RV \sim \lambda_k(M)
\quad\text{and}\quad
\lambda_k(M)\,\lambda_{q''-k+1}\bigl(M'\bigr) < \infty
$$
whenever $q' < k \le q''$. Putting both the relations together, for $k = q' + 1$
we particularly get
$$
\LV v_{q'+1}\RV\,\lambda_d\bigl(M'\bigr) < \infty\,.
$$
Realizing that the vectors $v_k$, for $q < k \le m$, are infinite, we see that
$\lambda_d\bigl(M'\bigr) \apr 0$. Thus there are vectors $w_1, \dots, w_d \in M'$,
linearly independent over $\sbR$ such that $\LV w_j\RV \le \lambda_d\bigl(M'\bigr)$
for $j \le d$; in particular, all the vectors $w_j$ are infinitesimal.

We will search for the vector $z'$ in the form
$$
z' = z + \alpha_1 w_1 + \ldots + \alpha_d w_d
$$
with unknown coefficients $\alpha_1,\dots,\alpha_d \in \FsR$. This will automatically
guarantee that $z' \apr z$.

As $\LV v_k\RV \ll \LV v_{q'+1}\RV$, for any $k \le q'$, $j \le d$, we have
$\LV v_k\RV \ll \LV v_{q'+1}\RV$ and
$$
\lv v_k\,w_j\rv \le \LV v_k\RV \LV w_j\RV \le
\frac{\LV v_k\RV}{\LV v_{q'+1}\RV}\,\LV v_{q'+1}\RV\,\lambda_d\bigl(M'\bigr)
\apr 0\,.
$$
At the same time, $v_k\,w_j \in \sbZ$, hence $v_k\,w_j = 0$, and
$$
v_k\,z' = v_k\,z + \sum_{j=1}^d \alpha_j v_k\,w_j = v_k\,z \in \sbZ\,,
$$
regardless of the choice of \,$\alpha_1,\dots\alpha_d$. Moreover, denoting
$h\:\sbR^n \to \sbR^d$ the $\sbR$-linear mapping given by
$H(\xi) = (\xi\,w_1, \dots, \xi\,w_d)^\T$ for $\xi \in \sbR^n$, we can conclude
that the vectors $v_1,\dots,v_{q'}$ form a basis of the linear subspace
$V \cap \Ker h \sbs \sbR^n$. Indeed, as the vectors $w_1,\dots,w_d$ are linearly
independent, $\Ker h$ has dimension $n - d$ and it equals the direct sum of the
orthocomplement $V^\perp$ with dimension $n - q''$ and $V \cap \Ker h$. Then
the latter necessarily has dimension $(n-d)-(n-q'') = q'$.

On the other hand, for $q' < k \le q''$, $j \le d$, we still have
$\LV v_k\RV \sim \LV v_{q'+1}\RV$ and
$$
\lv v_k\,w_j\rv \le \LV v_k\RV \LV w_j\RV \le
\frac{\LV v_k\RV}{\LV v_{q'+1}\RV}\,\LV v_{q'+1}\RV\,\lambda_d\bigl(M'\bigr)
< \infty\,,
$$
hence each $v_k\,w_j$ is a finite integer, and $h(v_k) \in \bZ^d$ for any $k$.
Since the vectors $v_1,\dots,v_{q'},v_{q'+1},\dots,v_{q''}$ are linearly independent
over $\sbR$ and the first $q'$ from among them form a basis of \,$V \cap \Ker\psi$,
the vectors $h(v_{q'+1}), \dots, h(v_{q''})$ are linearly independent over $\sbR$,
as well. Then the matrix $B = (b_{ij}) \in \sbR^{d \cx d}$ with entries
$b_{ij} = v_{q'+i}\,w_j \in \bZ$ satisfies $0 \napr\det B \in \bZ$. It follows that
$B$ is strongly regular and $B^{-1}$ is finite. Thus denoting
$\omega = (\omega_1,\dots,\omega_d)^\T \in \FsR^d$ the vector with coordinates
$\omega_j = v_{q'+i}\,z - \lfloor v_{q'+i}\,z\rfloor$ (i.e., the fractional parts
of the inner products $v_{q'+i}\,z$), for $i \le d$, the system $B\,\eta = -\omega$
has a unique solution
$\alpha = (\alpha_1, \dots, \alpha_d)^\T = -B^{-1}\,\omega\in \FsR^d$,
which means that
$$
\sum_{j=1}^d v_{q'+i}\,w_j\,\alpha_j = -\omega_i
$$
for each $i \le d$. Taking any $q' < k \le q''$ and putting $i = k - q'$, now, the
following computation
\begin{align*}
v_k\,z' &= v_k\,z + \sum_{j=1}^d \alpha_j v_{q'+i}\,w_j
= v_k\,z + \sum_{j=1}^d b_{ij} \alpha_j \\
&= v_{q'+i}\,z - \omega_i = \lfloor v_{q'+i}\,z\rfloor \in \sbZ
\end{align*}
concludes the proof of the Claim, henceforth of the Theorem, too.
\end{proof}

\begin{cor}\label{stand-part-annih-cor}
Let $L \sbs \sbR^n$ be an internal lattice. Then
$$
{}^\co(\Ansz L) = \Anz\bigl(\Lo\bigr)\,,
$$
in other words, the standard part of the internal integral annihilator $\Ansz L$
of \,$L$ equals the integral annihilator of the standard part $\Lo$ of \,$L$.
\end{cor}

\begin{proof}
The inclusion $\Anz(\Lo) \sbs {}^\co(\Ansz L)$ is a direct consequence of the
last Theorem. Indeed, if \,$x \in \Anz(\Lo)$ then $x\,\uo \in \bZ$ for every finite
$u \in L$. Then $\lv x\,u\rv_\bZ \apr 0$, for any such a $u$, and, by
Theorem~\ref{stand-part-annih}, there is a $y \in \Ansz(L)$, such that $y \apr x$,
hence $x \in {}^\co(\Ansz L)$.

The reversed inclusion ${}^\co(\Ansz L) \sbs \Anz(\Lo)$ is easy anyway. It suffices to
show that $\xo \in \Anz\bigl(\Lo\bigr)$ for any finite $x \in \Ansz(L)$. Taking any
finite $u \in L$, the inner product $u\,x$ is finite and belongs to $\sbZ$, hence
$$
\uo\,\xo = {}^\co(u\,x) = u\,x \in \bZ\,,
$$
so that $\xo \in \Anz\bigl(\Lo\bigr)$, as required.
\end{proof}

The following is the nonstandard formulation of the  announced ``almost-near''
result for dual lattices.

\begin{thm}\label{stab-dual-latt-ns}
Let  $L \sbs \sbR^n$ be an internal lattice. Then for each finite vector $x \in \sspn(L)$,
such that $\lv u\,x\rv_\bZ \apr 0$ for every finite $u \in L$, there is a $y \in L'$
such that $y \apr x$.
\end{thm}

\begin{proof}
Let $V = \sspn(L) \sbs \sbR^n$ and $z_V$ denote the orthogonal projection of any
$z \in \sbR^n$ to $V$. Then $\LV z_V\RV \le \LV z\RV$ for any~$z$. According to
Therorem~\ref{stand-part-annih}, under the above assumptions there is a
$y \in \Ansz(L)$ such that $y \apr x$. Then $v\,y_V = v\,y \in \sbZ$ for every
$v \in L$, i.e., $y_V \in L'$. As $x_V = x$ and $z \mapsto z_V$ is a linear map,
$$
\LV x - y_V \RV = \LV x_V - y_V \RV = \LV(x - y)_V\RV \le \LV x - y\RV \apr 0\,,
$$
hence $y_V \apr x$.
\end{proof}

The last stability Theorem is equivalent to the inclusion
$\bigl(\Lo\bigr)' \sbs {}^\co\bigl(L'\bigr)$ for internal lattices $L\sbs \sbR^n$.
In view of Corollary~\ref{stand-part-annih-cor} the reader might expect that also
the reversed inclusion ${}^\co\bigl(L'\bigr) \sbs \bigl(\Lo\bigr)'$ is satisfied
(and even easy to prove). However, as shown by following example, this is not true
in general.

\begin{ex}
Let $c \in \bR$ be positive and $d \in \sbR$ be positive and infinite. Consider the full
rank internal lattice
$$
L = c\,\sbZ \cx d\,\sbZ =
\bigl\{(a c, b d)^\T: a, b \in \sbZ\bigr\}
$$
in $\sbR^2$. Then, as easily seen, its standard part is a rank 1 lattice
$\Lo = c\,\bZ \cx \{0\}$ in $\bR^2$, while its internal dual is the full
rank internal lattice $L' = c^{-1} \sbZ \cx d^{-1} \sbZ$ in $\sbR^2$.
Then $\bigl(\Lo\bigr)' = c^{-1} \bZ \cx \{0\}$ is a rank~1 lattice in
$\bR^2$ while ${}^\co\bigl(L'\bigr) = c^{-1} \bZ \cx \bR$ is not even a lattice.
\end{ex}

\section{The ``almost-near'' theorem for dual lattices \\
standard formulation}\label{5}

\noindent
In this final section we state and prove the announced standard version of the stability
theorem for dual lattices, strengthening the preliminary Theorem~\ref{prelim}. It is in
fact a standard equivalent of Theorem~\ref{stab-dual-latt-ns}. In its proof we will need
the following last nonstandard lemma.

\begin{lem}\label{delta-apr0}
Let $L \sbs \sbR^n$ be an internal lattice and $G \sbs L$ be any additive subgroup of \,$L$.
Let further $\delta < \frac{1}{3}$ be a positive real number and $x \in \sbR^n$ be a vector
such that $\lv u\,x\rv_{\sbZ} \le \delta$ for every $u \in G$. Then
$\lv u\,x\rv_{\sbZ} \apr 0$ for every $u \in G$.
\end{lem}

\begin{proof}
As the mapping $u \mapsto u\,x$ is an additive group homomorphism $L \to \sbR$, the image
$G\,x = \{u\,x: u \in G\}$ of the subgroup $G \sbs L$ under this map must be a subgroup
of \,$\sbR$. However, if \,$0 < \delta < \frac{1}{3}$ is a (standard) real number, then
$\sbZ + \IsR$ is the biggest subgroup of \,$\sbR$ satisfying the inclusion
$\sbZ + \IsR \sbs \sbZ + {}^*[-\delta, \delta]$.
\end{proof}

Recall that $B = \{x \in \bR^n: \LV x\RV \le 1\}$ denotes the (euclidean) unit ball
in $\bR^n$.

\begin{thm}\label{stab-dual-latt-st}
Let $n \ge 1$ be an integer and $\delta < \frac{1}{3}$, $\eps$, $\lambda$ be positive reals.
Then there exists a real number $r > 0$, depending just on \,$n$, $\delta$, $\eps$ and
$\lambda$, such that every lattice $L \sbs \bR^n$, subject to $\lambda_1(L) \ge \lambda$, 
satisfies the following condition:

For any $x \in \spn(L)$, such that $\lv u\,x\rv_\bZ \le \delta$ for all $u \in L \cap r B$,
there is a $y \in L'$ such that $\LV x- y\RV \le \eps$.
\end{thm}

\begin{proof}
Assume that the conclusion of the Theorem fails for some fixed quadruple of admissible
parameters $n$, $\delta$, $\eps$, $\lambda$. This is to say that for each real number
$r > 0$ there is a lattice $L_r \sbs \bR^n$, satisfying $\lambda_1(L_r) \ge \lambda$,
and an $x_r \in \spn(L)$ such that $\lv u\,x\rv_\bZ \le \delta$ for every
$u \in L_r \cap r B$, however $\LV x_r - y\RV > \eps$ for any $y \in L'$, i.e.,
$(x_r + \eps B) \cap L_r' = \emptyset$. Let us confine to the values of \,$r$
from the set $I = \{1,2,3,\dots\}$ of all positive integers.

Let us pick any nonprincipal ultrafilter $D$ on the set $I$ and form the
ultraproduct $L = \prod_{r \in I} L_r\big{/}D$, as well as the vector
$x = (x_r)_{r\in I}\big{/}D \in L$ and the infinite positive hyperinteger
$\rho = (1,2,3,\dots)/D$. Then, by the virtue of the {\L}os Theorem (Lemma~\ref{Los}),
$L \sbs \sbR^n$ is an internal lattice satisfying $\lambda_1(L) \ge \lambda$.
For the same reason we have $x \in \sspn(L)$, $\lv u\,x\rv_{\sbZ} \le \delta$ for
every $u \in L \cap \rho\,\Bst$, as well as $(x + \eps\,\Bst) \cap L' = \emptyset$.
As $\FL = L \cap \FsR \sbs \rho\,\Bst$ and it is a subgroup of \,$L$, in view of
Lemma~\ref{delta-apr0} the second of the above three conditions implies that
$\lv u\,x\rv_{\sbZ} \apr 0$ for every $u \in \FL$.

As a consequence of Lemma~\ref{cov-rad-qual}, the covering radius
$\mu = \mu\bigl(L'\bigr)$ is a finite positive hyperreal. (In fact, Lemma~\ref{cov-rad}
and the {\L}os Theorem imply that $\mu \le n^{3/2}/(2\lambda)$, however, this is not
important for the moment.) Thus there is a $z \in L'$ such that $\LV x - z\RV \le \mu$.
Then $x - z \in \sspn(L)$ and
$$
u\,(x - z) - u\,x = -u\,z \in \sbZ\,,
$$
hence $\lv u\,(x - z)\rv_{\sbZ} = \lv u\,x\rv_{\sbZ}$ for any $u \in L$. At the same time,
$$
(x - z + \eps\,\Bst) \cap L' = (x + \eps\,\Bst) \cap L' = \emptyset\,.
$$
We can conclude, that $x' = x - z \in \sspn(L)$ is a finite vector satisfying
$\lv u\,x'\rv_{\sbZ} \apr 0$ for every finite $u \in L$, and $\LV x' - y\RV > \eps$
for any $y \in L'$. This, however, contradicts Theorem~\ref{stab-dual-latt-ns}.
\end{proof}

\begin{finrem}
Theorem~\ref{stab-dual-latt-ns} is rather robust in the sense that it does not explicitly
involve any norm on $\bR^n$ in its formulation. Moreover,
$$
\FsR^n = \{x \in \sbR^n: \LV x \RV < \infty\}
\quad\text{and}\quad
\IsR^n = \{x \in \sbR^n: \LV x \RV \apr 0\}
$$
for (the canonic extension to $\sbR^n$ of) any norm $\LV x\RV$ on $\bR^n$ and not just
for the euclidean one. As a consequence, Theorem~\ref{stab-dual-latt-st}, which is its
corollary, remains true even if \,$B$ denotes any centrally symmetric convex body in
$\bR^n$, $\lambda_1(L)$ is replaced by the first Minkowski successive minimum
$$
\lambda_1(C,L) = \min\bigl\{s \in \bR: s > 0,\ L \cap s\,C \ne \{0\}\bigr\}
$$
of another centrally symmetric convex body $C \sbs \bR^n$ with respect to $L$, and
$\LV x \RV$ is an arbitrary norm on $\bR^n$ (possibly without any direct relation
either to $B$ or to~$C$).
\end{finrem}

\end{document}